\documentclass[12pt]{amsart}

\usepackage{amsfonts}
\usepackage{amsthm}
\usepackage{amssymb}
\usepackage{amscd} 
\usepackage{amsmath}
\usepackage{mathtools}
\usepackage[top=30truemm,bottom=30truemm,left=30truemm,right=30truemm]{geometry}
\usepackage{indentfirst}
\newtheorem{thm}{Theorem}[section]
\newtheorem*{thmA}{Theorem A}
\newtheorem*{thmB}{Theorem B}
\newtheorem*{thmC}{Theorem C}
\newtheorem{lem}[thm]{Lemma}
\newtheorem*{lemA}{Lemma A}
\newtheorem{prop}[thm]{Proposition}
\newtheorem*{prop*}{Proposition}
\newtheorem{cor}[thm]{Corollary}

\theoremstyle{definition}

\newtheorem{dfn}[thm]{Definition}
\newtheorem*{dfn*}{Definition}
\newtheorem{exa}[thm]{Example}
\newtheorem{rmk}[thm]{Remark}

\newcounter{num}
\newcommand{\Rnum}[1]{\setcounter{num}{#1} \Roman{num}}

\begin{document}

\title[On Rigidity of Lattice of Topologies on Vector Spaces]{On Rigidity of Lattice of Topologies\\ on Vector Spaces}
\author{Takanobu Aoyama}
\address{Department of Mathematics, Graduate School of Science, Osaka University, 1-1 Machikaneyama-cho Toyonaka, Osaka 560-0043, Japan}
\email{t-aoyama@cr.math.sci.osaka-u.ac.jp}
\maketitle

\begin{abstract}
A vector topology on a vector space over a topological field is a (not necessarily Hausdorff) topology by which the addition and scalar multiplication are continuous. \par
We prove that, if an isomorphism between the lattice of topologies of two vector spaces preserves vector topologies, then the isomorphism is induced by a translation, a semilinear isomorphism and the complement map. As a consequence, if such an isomorphism exists, the coefficient fields are isomorphic as topological fields and these vector spaces have the same dimension. \par
We also prove a similar rigidity result for an isomorphism between the lattice of vector topologies which preserves Hausdorff vector topologies.\par
These results are obtained by using the fundamental theorems of affine and projective geometries.
\end{abstract}

\section{Introduction}
\subsection{Motivation and main results}
One of the interesting phenomenon in general topology is a space we consider may have many different topologies respecting the structure of the space. To compare two topologies, it is natural to consider which topology has more open subsets of the space than the other. The main object of this paper is the {\it lattice of topologies}, which is the set $\Sigma(X)$ of all topologies on an  arbitrary set $X$ endowed with the partial order by the inclusion order $\subset$. It is known that every pair of two elements of this partially ordered set has unique supremum and infimum, called {\it lattice structure} and it is studied by G. Birkhoff \cite{G.Bir}.
We focus on lattice isomorphisms between $\Sigma(X)$ and $\Sigma(Y)$, and on the group of lattice automorphisms denoted by ${\rm Aut}(\Sigma(X))$, motivated by a question to what extent the lattice structure of $\Sigma(X)$ determines $X$.
\par
Let us give two basic examples of lattice automorphisms between lattice of topologies.  The first example is the {\it induced map} $\theta_*:\Sigma(X)\rightarrow\Sigma(X)$ by an arbitrary bijection $\theta:X\rightarrow X$ defined by
\begin{align*}
\theta_*(T)=\{V\subset X\mid \theta^{-1}(V)\in T\},~T\in\Sigma(X).
\end{align*}
If the set $X$ is a finite set, the other example is the {\it complement map} $C:\Sigma(X)\rightarrow \Sigma(X)$ defined by
\begin{align*}
C(T)=\{X\setminus U\mid U\in T\},~T\in\Sigma(X).
\end{align*}
J. Hartmanis showed in \cite{J.Har},  that every lattice automorphism is $\theta_*$ or the composition $C\circ \theta_*$:
\begin{thm}[{\cite[Theorem 4]{J.Har}}]
For the lattice of topologies $\Sigma(X)$, we have 
\begin{itemize}
\item if the cardinality $|X|$ of $X$ is $1$,$2$ or infinite, then the group ${\rm Aut}(\Sigma(X))$ is isomorphic to the symmetric group of $X$, and
\item if $|X|$ is finite and more than $2$, then the group ${\rm Aut}(\Sigma(X))$ is isomorphic to the direct product of the symmetric group of $X$ and $2$-element group. 
\end{itemize} 
\end{thm}
Slight modifications of the proof in \cite{J.Har} of the above theorem shows that a lattice isomorphism between the lattice of topologies of two possibly  different sets must be given by a bijection and the complement map:
\begin{lemA}[Lemma \ref{Hartmanis}]
Let $X$ and $Y$ be sets. Suppose that there exists a lattice isomorphism $\Theta:\Sigma(X)\rightarrow\Sigma(Y)$ between their lattices of topologies. Then, there exists a unique bijection $\theta:X\rightarrow Y$ such that  
\begin{itemize}
\item if $|X|$ is $1,2$ or infinity, then $\Theta$ is equal to $\theta_*$, and 
\item if $|X|$ is finite and more than $2$, then $\Theta$ is equal to $\theta_*$ or $C\circ\theta_*$. 
\end{itemize}
\end{lemA}
An easy consequence of this lemma is that the lattice structure of $\Sigma(X)$ determines the cardinality of $X$ (Corollary \ref{Corollary of Hartmanis}). \par
One of the results of this paper is analogous to the above theorem in the linear algebraic setting: Let $X$ be a vector space over a topological field $K$. The dimension $\dim_KX$ of $X$ may be infinite. We denote by $\tau_K(X)$, the set of all topologies on $X$ such that $X$ becomes a topological vector space. This is the set of topologies on $X$ by which the addition and scalar multiplication are continuous. We call an element of $\tau_K(X)$ {\it vector topology} (Definition \ref{topological vector space}). Note that the poset $\tau_K(X)$ with the inclusion is also a lattice. 
In this setting, let $X$ and $Y$ be two vector spaces over topological fields $K$ and $L$, respectively, with the dimension of $X$ more than $1$. The main theorem asserts that a lattice isomorphism between $\Sigma(X)$ and $\Sigma(Y)$ which preserves the sets of vector topologies is induced by some composition of a semilinear isomorphism, a translation and $C$. Here, a {\it semilinear isomorphism} is a linear isomorphism which is allowed to twist a scalar multiplication by a fixed field isomorphism (Definition \ref{semilinear isomorphism}).

\begin{thmA}[Theorem \ref{main thm}]
Let $X$ and $Y$ be vector spaces over $K$ and $L$, respectively, with $\dim_K X\geq2$. Suppose that there is a lattice isomorphism  $\Phi:\Sigma(X)\rightarrow\Sigma(Y)$ with $\Phi(\tau_K(X))=\tau_L(Y)$. Then, there exist a unique triple $(\psi,\phi,y_0)$ consisting of an isomorphism $\psi:K\rightarrow L$ between the topological fields, $\psi$-semilinear isomorphism $\phi:X\rightarrow Y$ and an element $y_0\in Y$ such that
\begin{itemize}
\item if $|X|$ is infinity, then $\Phi=(\phi+y_0)_*$, and
\item if $|X|$ is finite, then $\Phi=(\phi+y_0)_*$ or $\Phi=C\circ(\phi+y_0)_*$.
\end{itemize}
\end{thmA}

Theorem A is reminiscent of the following two classical fundamental theorems of geometries, that is, the fundamental theorem of affine geometry and that of projective geometry.

\begin{thm}[The fundamental theorem of affine geometry]\label{F.T.A.G.}
Let $K,L$ be fields, $X$ be a vector space over $K$ with $\dim_K X\geq2$, and $Y$ be a vector space over $L$. If a bijection $\phi:X\rightarrow Y$ sends every two parallel lines of $X$ to parallel lines of $Y$, then $\phi$ is a semiaffine map. Namely, $\phi$ is a composition of semilinear isomorphism and a translation. 
\end{thm}

\begin{thm}[The fundamental theorem of projective geometry]\label{F.T.P.G.}
Let $K,L$ be fields, $X$ be a vector space over $K$ with $\dim_K X\geq 3$, and $Y$ be a vector space over $L$.
Let $\Phi$ be an isomorphism between lattices of all vector subspaces of $X$ and $Y$. Then, there exists a field isomorphism $\psi:K\rightarrow L$ and $\psi$-semilinear isomorphism $\phi:X\rightarrow Y$ such that $\Phi$ coincides with the induced map $\phi_*: S \mapsto \phi(S)$ for all subspaces $S$ of $X$. 
\end{thm}

We use the fundamental theorem of affine geometry in our proof of Theorem A.  Similarly, to prove Theorem C stated below, we apply the fundamental theorem of projective geometry. For these purposes, we connect a lattice of vector topologies and a lattice of subspaces by a Galois connection $\mathfrak{S},\mathfrak{T}$ (see Definition \ref{mathfrakS} and Definition \ref{mathfrakT}).\par

Let us discuss two easy consequence of Theorem A. One yields that the distribution of vector topologies $\tau_K(X)$ in the lattice of topologies $\Sigma(X)$ determines the structures of the coefficient field $K$ and linear structure of $X$ (Corollary \ref{cor main thm}).\\
The other consequence is on the group of lattice automorphisms on $\Sigma(X)$ which preserves $\tau_K(X)$, denote by ${\rm Aut}(\Sigma(X),\tau_K(X))$. This group is the subgroup of ${\rm Aut}(\Sigma(X))$ defined by
\begin{align*}
{\rm Aut}(\Sigma(X),\tau_K(X))=\{\Phi \in\rm{Aut}(\Sigma(X)) \mid \Phi(\tau_K(X))=\tau_K(X)\}.
\end{align*}
We also denote by $\Gamma L_h(X)$, the group of all semilinear transformations on $X$ whose adjiont field automorphisms of $K$ are also homeomorphism of $K$. 

\begin{thmB}[Theorem \ref{decompose}]
Let $X$ be a vector space over $K$ with $\dim_KX\geq2$. For the group ${\rm Aut}(\Sigma(X),\tau_K(X))$, we have the following:
\begin{itemize}
\item if $|X|$ is finite, then ${\rm Aut}(\Sigma(X),\tau_K(X))$ is isomorphic to $(X\rtimes \Gamma L_h(X))\times\mathbb{Z}/2\mathbb{Z}$, and \item if $|X|$ is infinite, then ${\rm Aut}(\Sigma(X),\tau_K(X))$ is isomorphic to $X\rtimes \Gamma L_h(X)$,
\end{itemize}
where the product $(x_1,\phi_1)\cdot(x_2,\phi_2)$ is defined by $(x_1+\phi_1(x_2),\phi_1\circ\phi_2)\in X\rtimes\Gamma L_h(X)$.
\end{thmB}

Next, because the restriction of $\Phi$ in Theorem A to the lattice of vector topologies $\tau_K(X)$ is a lattice isomorphism between $\tau_K(X)$ and $\tau_L(Y)$, it is natural to consider whether or not the same result holds when the assumption is only the existence of a lattice isomorphism between vector topologies. Although the same result does not hold in general (Example \ref{only compatible}), if the lattice isomorphism preserves Hausdorff vector topologies, we can prove a similar result by using the fundamental theorem of projective geometry. We denote by $\tau_K^H(X)$ and $\tau_L^H(Y)$, the set of all Hausdorff vector topologies on $X$ and $Y$, respectively.

\begin{thmC}[Theorem \ref{comatible topologies with Hausdroff}]
Let $K,L$ be topological fields, and let $X$ and $Y$ be vector spaces over $K$ and $L$, respectively, with $\dim_K(X)\geq3$. If there exists a lattice isomorphism $\Phi:\tau_K(X)\rightarrow\tau_L(Y)$ such that $\Phi(\tau_K^H(X))=\tau_L^H(Y)$, then fields $K$ and $L$ are algebraically isomorphic, and the vector spaces $X$ and $Y$ have the same dimension.
\end{thmC}

We give an example (Example \ref{not necessarily continuous}) to show that the field isomorphism between $K$ and $L$ in Theorem C is not necessarily continuous.
 
\subsection{Structure of this paper}
In Section 2, as a preparation, we recall some definitions related to lattices and vector topologies. Next, we see the fundamental theorems of geometries and construct a Galois connection between a lattice of vector topologies and a lattice of subspaces.  
In Section 3, we give a proof of slightly modified version of the result of J. Hartmanis to prove Theorem A. Then we prove Theorem A and obtain two applications: one is stating that the distribution of vector topologies determines its coefficient field and dimension. The other is on the group structure of lattice automorphisms which preserve vector topologies (Theorem B). In Section 4, we end the paper by proving Theorem C and giving an example. 
\section{Preliminaries}
\subsection{Preliminaries on lattice}
A partially ordered set (abbreviated to poset) $(\mathbb{P},\leq)$ is a {\it lattice} if every pair of two elements of $\mathbb{P}$ has the supremum (the least upper bound) and infimum (the greatest lower bound) with respect to the order $\leq$. By the antisymmetric law, the supremum and infimum are unique for each pair. We denote the supremum and infimum of $x,y\in\mathbb{P}$ by $x\vee y$ and $x\wedge y$, and call them the {\it join} and {\it meet} of $x,y$, respectively.\\
A lattice $(\mathbb{P},\leq)$ is said to be {\it complete} if any subsets $S$ of $\mathbb{P}$ have the supremum and infimum. We denote by $\bigvee S$ and $\bigwedge S$, the supremum and infimum of $S$, respectively. In particular, since $S=\emptyset\subset\mathbb{P}$ has supremum and infimum for a complete lattice $\mathbb{P}$, it has the minimum and maximum element.\\
A map $\Phi$ between two lattices $\mathbb{P}$ and $\mathbb{Q}$ is a {\it lattice homomorphism} if $\Phi$ preserves their joins and meets. By definition, the order is characterized by the join or meet, and thus every lattice homomorphism is an order preserving map, but the converse is not true in general. 
When a lattice homomorphism $\Phi:\mathbb{P}\rightarrow \mathbb{Q}$ has the inverse lattice homomorphism $\Phi^{-1}:\mathbb{Q}\rightarrow\mathbb{P}$, the map $\Phi$ is called {\it lattice isomorphism}. It is easy to see that a map is a lattice isomorphism if and only if it is an order preserving bijection, and that a lattice isomorphism between complete lattices $\mathbb{P}$ and $\mathbb{Q}$ preserves supremum and infimum of any subsets of $\mathbb{P}$. Two lattices are {\it isomorphic} if there exists a lattice isomorphism between them. A subset $S$ of a lattice $\mathbb{P}$ is a {\it sublattice} if the inclusion map $S\hookrightarrow \mathbb{P}$ is a lattice homomorphism.
\par
\begin{exa}
For a vector space, the poset of all vector subspaces with the inclusion order is a lattice. Here, the join $S_1\vee S_2$ is $S_1+S_2$ and the meet $S_1\wedge S_2$ is $S_1\cap S_2$ for subspaces $S_1,S_2$.
\end{exa}

\begin{dfn}
Let $X$ be a vector space over a field $K$. We denote by $\sigma_K(X)$, the lattice of $K$-vector subspaces with the inclusion order $\subset$.
\end{dfn}

\subsection{Lattice of topologies and vector topologies}
In this subsection, we recall the lattice of topologies.\\
Let $X$ be a non-empty set. We denote by $\Sigma(X)$, the poset consisting of all topologies on $X$ endowed with the order $\subset$. In this paper, we treat a topology $T\in\Sigma(X)$ as a family of open subsets of $X$. For a family of topologies $\{T_\lambda\}_{\lambda\in\Lambda}$ on $X$, the intersection $\bigcap_{\lambda\in\Lambda}T_{\lambda}$ and topology generated by $\bigcup_{\lambda\in\Lambda}T_{\lambda}$ are the infimum and supremum of the family $\{T_\lambda\}_{\lambda\in\Lambda}$, respectively, where we consider the intersection $\bigcap_{\lambda\in\Lambda}T_{\lambda}$ as the discrete topology if $\Lambda$ is empty. Thus, the poset $(\Sigma(X),\subset)$ is a complete lattice.\\
For a map $f:X\rightarrow Y$ between two sets, we can define induced maps $f_*:\Sigma(X)\rightarrow\Sigma(Y)$ and $f^*:\Sigma(Y)\rightarrow\Sigma(X)$ by
\begin{align*}
f_*(T)&=\{V\subset Y\mid f^{-1}(V)\in T\},~T\in\Sigma(X),\\
f^*(T')&=\{f^{-1}(V)\subset X\mid V\in T'\},~T'\in\Sigma(Y). 
\end{align*}
\begin{dfn}\label{topological field}
Let $K$ be a field endowed with a Hausdorff topology. The field $K$ is a {\it topological field} if the three operators of $K$
\begin{itemize}
\item{additive operator}: $K\times K \ni (x,y) \mapsto x+y \in K$,
\item{multiple operator}: $K\times K \ni (x,y) \mapsto xy \in K$,
\item{inverse operator}: $K\setminus\{0\} \ni x \mapsto x^{-1} \in K\setminus\{0\}$
\end{itemize}
are continuous, where $K\times K$ and $K\setminus\{0\}$ are endowed with the product topology and relative topology of $K$, respectively. 
\end{dfn}

Here, if we drop the assumption of Hausdorff topology on a topological field $K$, the topological field $K$ only has the trivial indiscrete topology $\{\emptyset,K\}$.

\begin{dfn}\label{topological vector space}
A vector space $X$ endowed with a (possibly non-Hausdorff) topology over a topological field $K$ is called {\it topological vector space} if the two operators 
\begin{itemize}
\item{additive operator}: $X\times X \ni (x,y) \mapsto x+y \in X$,
\item{scalar multiple operator}: $K\times X \ni (\alpha,x) \mapsto \alpha\cdot x \in X$
\end{itemize}
are continuous, where $X\times X$ and $K\times X$ are endowed with the product topologies. We call the topology on the topological vector space $X$ {\it vector topology}. \\
We denote by $\tau_K(X)$, the poset of all vector topologies on $X$ with inclusion $\subset$, and denote by $\tau_K^H(X)$, the subset of $\tau_K(X)$ consisting of Hausdorff vector topologies.
\end{dfn}

It is known that the set $\tau_K(X)$ is a complete lattice. Thus $\tau_K(X)$ has the maximum element (the strongest topology), which we denote by $T^{\max}_X$. We show in the next subsection, that $T^{\max}_X\in \tau_K^H(X)$. 
Although $\tau_K(X)$ is a subposet of $\Sigma(X)$, the lattice $\tau_K(X)$ is not a sublattice of $\Sigma(X)$ in general. Namely, the meet of $\tau_K(X)$ does not coincide with that of $\Sigma(X)$. 

\subsection{Semilinear isomorphism}
We recall semilinear isomorphism in this subsection.
\begin{dfn}\label{semilinear isomorphism}
Let $X$ and $Y$ be vector spaces over fields $K$ and $L$, respectively. For a field isomorphism $\psi:K\rightarrow L$, a bijection $\phi:X\rightarrow Y$ is a $\psi$-{\it semilinear isomorphism} if $\phi$ satisfies
\begin{align*}
\phi(x_1+x_2)&=\phi(x_1)+\phi(x_2)~\text{for}~x_1,x_2\in X, \text{and}\\
\phi(\alpha\cdot x)&=\psi(\alpha)\cdot\phi(x)~\text{for}~\alpha\in K,x\in X.
\end{align*}
We simply call $\phi$ a {\it semilinear isomorphism} if $\phi$ is a $\psi$-semilinear isomorphism for some field isomorphism $\psi$.
\end{dfn}
If $\psi:K\rightarrow L$ is a field isomorphism and also a homeomorphism, a $\psi$-semilinear map $\phi:X\rightarrow Y$ induces maps between $\tau_K(X)$ and $\tau_L(Y)$ by restricting $\phi_*$ and $\phi^*$. 
In the case when $K=L$ and $\phi:X\rightarrow Y$ is a linear map, the same holds even if we drop the injectivity of $\phi$ for $\phi_*$, and the injectivity and surjectivity of $\phi$ for $\phi^*$. 
By using this fact, we see that $T^{\max}_X$ is a Hausdorff topology, that is, $T^{\max}_X\in\tau_K^H(X)$. Let $x_1,x_2$ be two different elements of $X$. We denote by $\langle x_1-x_2\rangle$, a 1-dimensional subspace generated by $x_1-x_2$, and decompose $X$ into a direct sum $X=X'\oplus \langle x_1-x_2\rangle$ for a subspace $X'$. We define a map $\phi:X\rightarrow K$ by
\begin{align*}
\phi: X'\oplus \langle x_1-x_2\rangle \ni x'+\alpha(x_1-x_2)\mapsto \alpha \in K.
\end{align*} 
It is clear that $K$ is itself a Hausdorff topological vector space, and thus a vector topology on $X$ that is an image of the topology of $K$ by $\phi^*$ has disjoint neighborhoods $U$ and $V$ of $0$ and $x_1-x_2$, respectively. Since $T^{\max}_X$ is the maximum topology, $U$ and $V$ belong to $T^{\max}_X$, and $U+x_2,~V+x_2$ are disjoint neighborhoods of $x_2$ and $x_1$ with respect to $T^{\max}_X$. Thus, we obtain:
\begin{prop}
$T^{\max}_X$ is a Hausdorff vector topology.
\end{prop}
\par
Lastly, we note that a semilinear isomorphism $\phi$ preserves the linearly independence. Thus $X$ and $Y$ have the same dimension if a semilinear isomorphism exists.

\subsection{Fundamental theorems of two geometries}

The fundamental theorems of affine and projective geometries are stated and generalized in various way. In this subsection, we recall one of them needed later. \par

Recall that two subsets $A_1,A_2$ of a vector space $X$ are called {\it parallel lines} if there are two elements $x_1,x_2\in X$ and one-dimensional subspace $l\subset X$ such that $A_1=x_1+l$ and $A_2=x_2+l$ holds. Then the fundamental theorem of affine geometry is stated as follows:

\begin{thm}[The fundamental theorem of affine geometry]\label{lemmaFTAG}
Let $K$ and $L$ be fields and let $X$ and $Y$ be vector spaces over $K$ and $L$, respectively, with $\dim_KX\geq2$. Then, every bijection $\phi:X\rightarrow Y$ which maps parallel lines of $X$ to that of $Y$ is a semiaffine map. 
\end{thm}

For a proof of Theorem \ref{lemmaFTAG}, see \cite[Fundamental Theorem of Affine Geometry, p. 12]{J.B.San} or \cite{P.Sch}. Although, the paper \cite{P.Sch} treats the case when $K=L$ and $X=Y$, the proof also works for Theorem \ref{lemmaFTAG}.

\begin{thm}[The fundamental theorem of projective geometry]\label{lemmaFTPG}
Let $K$ and $L$ be fields and let $X$ and $Y$ be vector spaces over $K$ and $L$, respectively, with $\dim_KX\geq3$. Then, every lattice isomorphism $\Phi:\sigma_K(X)\rightarrow\sigma_L(Y)$ is induced by a $\psi$-semilinear isomorphism $\phi:X\rightarrow Y$, where $\psi:K\rightarrow L$ is a field isomorphism.
\end{thm}
\begin{rmk}\label{remarkFTPG}
The proof of Theorem \ref{lemmaFTPG} in \cite[pp. 44--50]{R.Bae} is still valid for a lattice isomorphism between the sublattices of $\sigma_K(X),\sigma_L(Y)$ consisting of finite-dimensional subspaces; it is also induced by a semilinear isomorphism between $X$ and $Y$.
\end{rmk}

\subsection{Galois connection}

We introduce the maps $\mathfrak{S},\mathfrak{T}$ to connect the lattice of vector topologies and the lattice of subspaces, which form geometries on vector spaces.
\begin{dfn}\label{mathfrakS}
Let $(X,T)$ be a topological vector space. We define a subset $\mathfrak{S}_X(T)$ of $X$ by
\begin{align*}
\mathfrak{S}_X(T)=\bigcap_{0\in U \in T}U.
\end{align*}
\end{dfn}

We see that $\mathfrak{S}_X(T)$ is a vector subspace of $X$. By definition, $0$ belongs to $\mathfrak{S}_X(T)$. For two elements $x,y$ from $\mathfrak{S}_X(T)$ and an open neighborhood $U$ of zero, by the continuity of the addition at $(0,0)\in X\times X$, there exists an open neighborhood $V$ of zero such that $V+V\subset U$ holds. Since $x,y$ belong to $V$ by definition of $\mathfrak{S}_X(T)$, the sum $x+y$ belongs to $U$. Thus, $\mathfrak{S}_X(T)$ is closed under the addition. For $\alpha\in K$ and $x\in \mathfrak{S}_X(T)$, let $U$ be an arbitrary neighborhood of zero. By the continuity of scalar multiple at $(\alpha,0)\in K\times X$, there exists a neighborhood $V$ of zero such that $\alpha\cdot V\subset U$. By definition of $\mathfrak{S}_X(T)$, the element $x$ belongs to $V$, which implies $\alpha\cdot x\in U$. Thus, $\mathfrak{S}_X(T)$ is closed under the scalar multiple. Therefore, the subset $\mathfrak{S}_X(T)$ is a $K$-vector subspace for any $T\in \tau_K(X)$, and $\mathfrak{S}_X$ is a map from $\tau_K(X)$ to $\sigma_K(X)$.

\begin{dfn}\label{mathfrakT}
Let $S$ be a vector subspace of a vector space $X$. We define a topology $\mathfrak{T}_X(S)$ by
\begin{align*}
\mathfrak{T}_X(S)=\{U+S \mid U\in T^{\max}_X\},
\end{align*}
where the topology $T^{\max}_X$ is the maximum element (the strongest topology) in ${(\tau_K(X),\subset)}$.
\end{dfn}
The topology $\mathfrak{T}_X(S)$ coincides with the topology defined by ${\pi_S}^*\circ{\pi_S}_*(T^{\max}_X)$ and this is a vector topology, where $\pi_S:X\rightarrow X/S$ is the natural quotient map. Therefore, $\mathfrak{T}_X$ is a map from $\sigma_K(X)$ to $\tau_K(X)$.

\begin{lem}\label{composition}
The following {\rm(1)}, {\rm(2)} and {\rm(3)} hold true.\\
\begin{enumerate}
\item[\rm(1)]~The composition $\mathfrak{S}_X\circ \mathfrak{T}_X$ is the identity map on the lattice $\sigma_K(X)$ of subspaces.
\item[\rm(2)]~For every $T\in\tau_K(X)$, we have $T\subset \mathfrak{T}_X\circ \mathfrak{S}_X(T)$. 
\item[\rm(3)]~For $T\in\tau_K(X)$, the subspace $\mathfrak{S}_X(T)$ is 0-dimensional if and only if $T$ is a Hausdorff topology.
\end{enumerate}
\end{lem}

\begin{proof}
(1)~Let $S$ be a subspace of $X$. We show that
\begin{align*}
\bigcap_{0\in V \in T^{\max}_X}(V+S)=\bigcap_{0\in U+S,~U\in T^{\max}_X}(U+S)=\mathfrak{S}_X\circ \mathfrak{T}_X(S).
\end{align*}
The second equality follows from definitions of $\mathfrak{S}_X$ and $\mathfrak{T}_X$. For the first equality, since 0 belongs to $S$, the set $V+S$ contains 0 for every $0\in V\in T^{\max}_X$, which implies that $\bigcap_{0\in V \in T^{\max}_X}(V+S)\supset\bigcap_{0\in U+S,~U\in T^{\max}_X}(U+S)$. For the other inclusion, let $U+S=\bigcup_{x\in S}U+x~(U\in T^{\max}_X)$ be an open neighborhood of zero with respect to $T^{\max}_X$. By the continuity of the addition at $(0,0)\in X\times X$, there exists an open neighborhood $V \in T^{\max}_X$ of zero such that $V\subset U+S$. Since $S$ is a subspace, $V+S$ is also included in $U+S$. Therefore, we have $\bigcap_{0\in V \in T^{\max}_X}(V+S)=\mathfrak{S}_X\circ \mathfrak{T}_X(S)$.\\
By definition, $\bigcap_{0\in V\in T^{\max}_X}V+S$ contains $S$, which implies $S\subset \mathfrak{S}_X\circ \mathfrak{T}_X(S)$. Before we prove the other inclusion, we show that $\pi_*(T^{\max}_X)=T^{\max}_{X/S}$, where $\pi:X\rightarrow X/S$ is the quotient map. 
Since $\pi_*$ and $\pi^*$ send vector topologies to vector topologies, by definition of the maximum elements, $\pi_*(T^{\max}_X)\subset T^{\max}_{X/S}$ and $\pi^*(T^{\max}_{X/S})\subset T^{\max}_X$ holds. From the latter inclusion, $\pi^{-1}(V)\in T^{\max}_X$ holds for every $V\in T^{\max}_{X/S}$, which implies that $T^{\max}_{X/S}\subset \pi_*(T^{\max}_X)$ holds.  
Now, assume that we can take a point $x$ from $\mathfrak{S}_X\circ \mathfrak{T}_X(S)\setminus S$. Since $(X/S,T^{\max}_{X/S})$ is a Hausdorff space, there are disjoint open neighborhoods $V_1,V_2\in T^{\max}_{X/S}=\pi_*(T^{\max}_X)$ of $\pi(x)$ and $\pi(0)$, respectively. Since ${\rm Ker}~\pi=S$, we have $\pi^{-1}(V_2)+S=\pi^{-1}(V_2)\in T^{\max}_X$, which implies that the set $\pi^{-1}(V_2)$ is an open neighborhood of zero in $\mathfrak{T}_X(S)$. 
Thus, $x$ belongs to $\mathfrak{S}_X\circ \mathfrak{T}_X(S) \subset \pi^{-1}(V_2)$ and $\pi^{-1}(V_1)$. This contradicts that $\pi^{-1}(V_1)$ and $\pi^{-1}(V_2)$ are disjoint.\par
(2)~We denote by $S$, the subspace $\mathfrak{S}_X(T)$. For an open subset $V\in T$, we show that $V+S=V$ holds. Since $S$ contains zero, $V+S\supset V$ is clear. For the other inclusion, we take $v+s,~v\in V,s\in S$ from $V+S$. By the continuity of the addition at $(v,0)\in X\times X$ with respect to $T$, there are two open subsets $U_1\ni v,~U_2\ni 0$ such that $U_1+U_2\subset V$. By definition of $\mathfrak{S}_X(T)$, the element $s$ belongs to $U_2$, and thus $v+s\in V$. Since $T^{\max}_X$ contains $T$, the set $V=V+S$ belongs to $\mathfrak{T}_X\circ\mathfrak{S}_X(T)$.\par
(3)~Assume that $\mathfrak{S}_X(T)$ is 0-dimensional. Let $x_1,x_2$ be two different elements of $X$. Since $x_1-x_2\not=0$ does not belong to $\mathfrak{S}_X(T)$, we have an open neighborhood $U$ of zero with $x_1-x_2\not\in U$. By the continuity of the map $X\times X\ni (x,x')\mapsto x-x'\in X$, there is an open neighborhood $V\in T$ of zero such that $V-V\subset U$. Then, $x_1+V$ and $x_2+V$ are disjoint open neighborhoods of $x_1$ and $x_2$, respectively. \\
Assume that $(X,T)$ is a Hausdorff space. Then, for every non-zero element $x\in X$, there is an open neighborhood $U$ of zero with $x\not\in U$. By definition, $U$ contains $\mathfrak{S}_X(T)$, and thus $x\not\in\mathfrak{S}_X(T)$.
\end{proof}

\begin{rmk}
By definition, $\mathfrak{S}_X:\tau_K(X)\rightarrow\sigma_K(X)$ and $\mathfrak{T}_X:\sigma_K(X)\rightarrow\tau_K(X)$ reverse the inclusion. Combining with the result of Lemma \ref{composition}, the pair of maps $(\mathfrak{S}_X,\mathfrak{T}_X)$ is an {\it antitone Galois connection}. \\
Here, an {\it antitone Galois connection} is a pair $(f,g)$ of two maps $f:\mathbb{P}\rightarrow\mathbb{Q}$ and $g:\mathbb{Q}\rightarrow\mathbb{P}$ between two posets $(\mathbb{P},\leq_{\mathbb{P}})$ and $(\mathbb{Q},\leq_{\mathbb{Q}})$ such that
\begin{itemize}
\item{}~$f$ and $g$ reverse the orders, and
\item{}~$q\leq_{\mathbb{Q}} f(p)$ if and only if $p\leq_{\mathbb{P}} g(q)$ for $p\in\mathbb{P}$ and $q\in \mathbb{Q}$
\end{itemize}
hold.
\end{rmk}

\section{Lattice of topologies and lattice of vector topologies}

In this section, we prove the following main theorem.
\begin{thm}\label{main thm}
Let $K,L$ be topological fields and $X,Y$ be vector spaces over $K$ and $L$, respectively, with $\dim_K X\geq2$. Suppose that there is a lattice isomorphism $\Phi:\Sigma(X)\rightarrow \Sigma(Y)$ such that $\Phi(\tau_K(X))$ coincides with $\tau_L(Y)$. Then, there exists a unique triple $(\psi,\phi,y_0)$ such that
\begin{itemize}
\item if the cardinality of $X$ is infinite, then the induced map $(\phi+y_0)_*$ is equal to $\Phi$, and
\item if the cardinality of $X$ is finite, then the induced map $(\phi+y_0)_*$ is equal to $\Phi$ or $C\circ \Phi$,
\end{itemize}
where $\psi:K\rightarrow L$ is an isomorphism between topological fields, $\phi:X\rightarrow Y$ is a $\psi$-semilinear isomorphism 
and $y_0$ is an element of $Y$.
\end{thm}

We have the following immediate corollary of Theorem \ref{main thm}.
\begin{cor}\label{cor main thm}
If $\Sigma(X)$ is isomorphic to $\Sigma(Y)$ by a lattice isomorphism whose restriction to $\tau_K(X)$ is a lattice isomorphism between $\tau_K(X)$ and $\tau_L(Y)$, then two topological fields $K,L$ are isomorphic and $X,Y$ are semilinear isomorphic. In particular, the dimensions of $X$ over $K$ and $Y$ over $L$ are equal.
\end{cor}

\subsection{Result of J. Hartmanis}
In our proof of Theorem \ref{main thm}, we need the following lemma, which is essentially due to J. Hartmanis \cite[Theorem 4]{J.Har}. 
Although the original statement is on the group of lattice automorphisms of $\Sigma(X)$, the proof is valid for the case of lattice isomorphisms between different sets.
\begin{lem}\label{Hartmanis}
Let $X,Y$ be non-empty sets. Suppose that there is a lattice isomorphism $\Theta: \Sigma(X)\rightarrow\Sigma(Y)$. Then, there exists a unique bijection $\theta:X\rightarrow Y$ such that
\begin{itemize}
\item if the cardinality of $X$ is one, two or infinity, then the induced map $\theta_*$ is equal to $\Theta$, and
\item if the cardinality of $X$ is finite and more than two, then the induced map $\theta_*$ is equal to $\Theta$ or $C\circ\Theta$.
\end{itemize}
\end{lem}
Since it is clear that a bijection between two sets $X$ and $Y$ induces a lattice isomorphism between $\Sigma(X)$ and $\Sigma(Y)$, we obtain an immediate corollary of Lemma \ref{Hartmanis}:
\begin{cor}\label{Corollary of Hartmanis}
Two lattices $\Sigma(X)$ and $\Sigma(Y)$ are isomorphic if and only if $X$ and $Y$ have the same cardinality.
\end{cor}
We give a proof according to \cite[Theorem 4]{J.Har}. 
Before the proof, we recall some notions and their properties. An element $a$ of a lattice $L$ with the minimum element $0$ is called an {\it atom} if $a$ is not $0$ and there are no elements between them. Namely, $a$ is the next weakest element to $0$. 
Given a set $X$ with $|X|\geq2$, every atom of lattice of topologies can be written as 
\begin{align*}
A_X(D)=\{\emptyset,D,X\},
\end{align*}
where $D$ is a subset $\emptyset\subsetneq D\subsetneq X$. 
We denote by $\mathfrak{p}_X$, the set of all atoms of $\Sigma(X)$. 
We abbreviate $A_X(D)$ and $\mathfrak{p}_X$ to $A(D)$ and $\mathfrak{p}$, respectively, if the underlying set $X$ is clear. 
For a point $x\in X$, we denote atoms $A(\{x\})$ and $A(X\setminus \{x\})$ by $A(x)$ and $A(x^c)$, respectively. \\
Recall that a complete lattice is said to be {\it atomic} if every element can be written as a supremum of a set of atoms. Now, let $T$ be an arbitrary topology on $X$. Then, we define a topoplogy $T'$ as a supremum of a set of atoms by
\begin{align*}
T'=\bigvee_{D\in T,\emptyset\subsetneq D\subsetneq X}A(D).
\end{align*}
It is clear that every $A(D)$ is weaker than $T$ and $T'$ contains all open subsets of $T$. Thus, two topologies $T$ and $T'$ coincide. Therefore, the lattice of topologies $\Sigma(X)$ is atomic, including the case when $|X|=1$.\\
When the cardinality of $X$ is more than 2, we decompose $\mathfrak{p}$ into three disjoint subsets $\mathfrak{n},\mathfrak{m}$ and $\mathfrak{l}$ defined by
\begin{align*}
\mathfrak{n}&=\{A(x)\mid x\in X\},\\
\mathfrak{m}&=\{A(x^c)\mid x\in X\},\\
\mathfrak{l}&=\mathfrak{p}\setminus(\mathfrak{n}\cup\mathfrak{m}). 
\end{align*}
Now, we recall that the function $t:\mathfrak{p}\times\mathfrak{p}\rightarrow \mathbb{N}$, called {\it type} in \cite{J.Har} is defined by the number of atoms weaker than or equal to the join $p\vee q$ for a pair of atoms $(p,q)$. 
Let $p=A(D_p)$ and $q=A(D_q)$ be two atoms. Then, the join $p\vee q$ can be written explicitly as
\begin{align*}
\{\emptyset, D_p\cap D_q, D_p,D_q, D_p\cup D_q, X\}.
\end{align*}
Thus, the type function takes values at most 4. We see more precisely what values the type takes with respect to the decomposition $\mathfrak{p}=\mathfrak{n}\sqcup\mathfrak{m}\sqcup\mathfrak{l}$. Let $p,q\in\mathfrak{n}$ be two different atoms and let $x_p,x_q\in X$ be two points such that $A(x_p)=p, A(x_q)=q$. The join $p\vee q$ is $\{\emptyset,\{x_p\}, \{x_q\}, \{x_p,x_q\}, X\}$, which is stronger than three atoms $A(x_p),A(x_q), A(\{x_p,x_q\})$.  Thus, the type $t(p,q)$ is 3 in this case. By continuing similar arguments, we obtain the following Table 1.
\begin{table}[h]
  \centering
  \begin{tabular}{c||c|c|c}
     & $p\in\mathfrak{n}$ &$p\in \mathfrak{m}$ &$p\in\mathfrak{l}$\\\hline \hline
    $q\in\mathfrak{n}$   & 3 & 2 	&2 or 3\\\hline
    $q\in\mathfrak{m}$  & 2 & 3&2 or 3 \\ \hline
    $q\in\mathfrak{l}$    &2 or 3& 2 or 3&2 or 3 or 4  
  \end{tabular}
  \caption{possible values of the type $t(p,q)$ with $p\not=q$} 
  \label{type}
\end{table}

Moreover, when $4\leq |X|$, for every $p=A(D)\in\mathfrak{l}$, there is an atom $q=A(D')\in\mathfrak{l}$ such that $t(p,q)$ is actually $4$ by taking $D'=\{x_1,x_2\},~x_1\in D,~x_2\in X\setminus D$. 

\begin{proof}[Proof of Lemma \ref{Hartmanis}]
Since atoms are characterized by the order, the restriction of the order preserving isomorphism $\Theta$ to the set of atoms $\mathfrak{p}_X$ induces a bijection between $\mathfrak{p}_X$ and $\mathfrak{p}_Y$, thus $\mathfrak{p}_X$ and $\mathfrak{p}_Y$ have the same cardinality. The set of atoms $\mathfrak{p}_X$ is a finite set if and only if $X$ is finite. Moreover, $|\mathfrak{p}_X|$ is equal to $|\mathcal{P}(X)|-2$ if $X$ is finite, where $\mathcal{P}(X)$ denotes the set of all subsets of $X$. Thus, when $X$ is a finite set, $Y$ is also a finite set and their cardinalities are equal.   
In particular, if the cardinality of $X$ is $1$, then that of $Y$ is $1$. In this case, the claim of this lemma obviously holds. If $|X|$ is $2$, let $x_1,x_2$ and $y_1,y_2$ be two points of $X$ and $Y$, respectively. Then, the sets of atoms of the lattices $\Sigma(X)$ and $\Sigma(Y)$ are $\mathfrak{p}_X=\{ A_X(x_i) \mid i=1,2\}$ and $\mathfrak{p}_Y=\{A_Y(y_i)\mid i=1,2\}$, respectively. Since the restriction of $\Theta$ induces a bijection between $\mathfrak{p}_X$ and $\mathfrak{p}_Y$, we can define a bijection $\theta:X\rightarrow Y$ so that $\Theta(A_X(x_i))=A_Y(\theta(x_i))$, which induces the map $\theta_*=\Theta$. The uniqueness of the map $\theta$ is trivial. \par
For the rest of this proof, we assume that $|X|$ is more than two. 
Now, we show that either $\Theta(\mathfrak{n}_X)=\mathfrak{n}_Y,~\Theta(\mathfrak{m}_X)=\mathfrak{m}_Y$ or $\Theta(\mathfrak{n}_X)=\mathfrak{m}_Y,~\Theta(\mathfrak{m}_X)=\mathfrak{n}_Y$ holds. 
Assume that $\Theta$ sends an element $p$ of $\mathfrak{n}_X\cup\mathfrak{m}_X$ to $\mathfrak{l}_Y$. Then, there exists $\Theta(q)\in\mathfrak{l}_Y$ such that $t(\Theta(p),\Theta(q))=4$. It is clear that the lattice isomorphism $\Theta$ preserves the type $t$, and thus $t(p,q)$ is 4, which contradicts the type $t(p,q)$ is at most 3 for any $q$ (see Table \ref{type}). Therefore, we obtain $\Theta(\mathfrak{n}_X\cup\mathfrak{m}_X)\subset\mathfrak{n}_Y\cup\mathfrak{m}_Y$ and the same argument for $\Theta^{-1}$ shows that this inclusion is equality.  
Next, assume that two distinct elements $p,q\in\mathfrak{n}_X$ satisfy $\Theta(p)\in\mathfrak{n}_Y,\Theta(q)\in\mathfrak{m}_Y$. Then, the type $t(\Theta(p),\Theta(q))=2$ by Table \ref{type}. On the other hands, $t(p,q)=3$ since both $p,q$ belong to $\mathfrak{n}_X$. This is a contradiction since $\Theta$ preserves the type. Thus either $\Theta(\mathfrak{n}_X)\subset \mathfrak{n}_Y$ or $\Theta(\mathfrak{n}_X)\subset \mathfrak{m}_Y$ holds. A similar argument also shows that either $\Theta(\mathfrak{m}_X)\subset\mathfrak{m}_Y$ or $\Theta(\mathfrak{m}_X)\subset\mathfrak{n}_Y$ holds. Combining the above three arguments, we conclude that either of the following holds.
\begin{itemize}
\item[(1)]~$\Theta(\mathfrak{n}_X)=\mathfrak{n}_Y$ and $\Theta(\mathfrak{m}_X)=\mathfrak{m}_Y$,
\item[] or
\item[(2)]~$\Theta(\mathfrak{n}_X)=\mathfrak{m}_Y$ and $\Theta(\mathfrak{m}_X)=\mathfrak{n}_Y$.
\end{itemize}
First, we consider the case of (1). Since $\Theta(\mathfrak{n}_X)=\mathfrak{n}_Y$, we can define a bijection $\theta:X\rightarrow Y$ so that $\Theta(A(x))=A(\theta(x))$ holds for every $x\in X$. Now let $D$ and $D'$ be proper subsets of $X$ and $Y$, respectively such that $\Theta(A(D))=A(D')$ holds. Since the lattice isomorphism $\Theta$ preserves the inequality $A(D)\subset \bigvee_{x\in D}A(x)$, we have $A(D') \subset \bigvee_{x\in D} A(\theta(x))$. The supremum $\bigvee_{x\in D}A(\theta(x))$ is $\mathcal{P}(\theta(D))\cup\{Y\}$, which implies that the subset $\theta(D)$ contains $D'$. Combining with the same argument for $\Theta^{-1}$ and $\theta^{-1}$ we obtain  $D'=\theta(D)$. Therefore, the two maps $\Theta$ and $\theta_*$ coincide when they are restricted to the atoms $\mathfrak{p}_X$, which implies $\Theta=\theta_*$ since $\Sigma(X)$ is an atomic lattice. \\
Next, for the case of (2), the map $\Theta$ sends the topology $\bigvee\mathfrak{n}_X$ to $\bigvee\mathfrak{m}_Y$. When the cardinality of $Y$ is infinity, these supremums are 
\begin{align*}
&\bigvee\mathfrak{n}_X=\mathcal{P}(X), \text{and}\\
&\bigvee\mathfrak{m}_Y=\{Y,\emptyset\}\cup\{Y\setminus F,~F:\text{finite}\},
\end{align*}
and thus the case (2) does not happen. 
When $|Y|=|X|$ is finite, the map $C:\mathcal{P}(Y)\rightarrow\mathcal{P}(Y)$ is a lattice isomorphism when we restrict $C$ to $\Sigma(Y)$. Then, we have $C\circ \Theta(\mathfrak{n}_X)=\mathfrak{n}_Y$ and $C\circ\Theta(\mathfrak{m}_X)=\mathfrak{m}_Y$. Thus, by applying the argument of (1) for the composition $C\circ\Theta$, we obtain the map $\theta$ which satisfies the claim of this lemma.\par
Lastly, we show the uniqueness of $\theta$. Let $\theta_1,\theta_2:X\rightarrow Y$ be bijections satisfying the condition of the lemma.
By definition, every induced map $\theta_*$ by a bijection $\theta$ sends $A_X(D)$ to $A_Y(\theta(D))$. Thus, we have $|D|=|\theta(D)|$. This implies that $\theta_*$ sends every topology of $\mathfrak{n}_X$ to that of $\mathfrak{n}_Y$, and the induced map does not coincide with $C$ since $|Y|$ is more than 2. Therefore,   
we have ${\theta_1}_*=\Theta={\theta_2}_*$ or ${\theta_1}_*=C\circ\Theta={\theta_2}_*$. Therefore, for every $x\in X$, the equality $A(\theta_1(x))={\theta_1}_*(A(x))={\theta_2}_*(A(x))=A(\theta_2(x))$ holds, which implies $\theta_1(x)=\theta_2(x)$. 
\end{proof}
\subsection{Proof of Theorem \ref{main thm} and corollaries}
\begin{proof}[Proof of Theorem \ref{main thm}]
By Lemma \ref{Hartmanis}, there exists a bijection $\phi_0:X\rightarrow Y$ such that ${\phi_0}_*$ is $\Phi$ or $C\circ \Phi$.\par
We first consider the case of ${\phi_0}_*=\Phi$. We set $y_0=\phi_0(0)$ and $\phi=\phi_0-y_0$ so that $\phi(0)=0$. It suffices to show that $\phi$ is induced by a $\psi$-semilinear map, where $\psi:K\rightarrow L$ is an isomorphism between the topological fields. The translation map $y\mapsto y+y_0$ is a homeomorphism from $(Y,T)$ to itself for every $T\in\tau_L(Y)$. Thus, the image of $\tau_K(X)$ by the induced map $\phi_*:\Sigma(X)\rightarrow \Sigma(Y)$ is $\tau_L(Y)$. \\
We show that the image of a subspace $S$ of $X$ by $\phi$ is a subspace of $Y$. By Lemma \ref{composition} and Definition \ref{mathfrakS}, the image $\phi(S)=\phi(\mathfrak{S}_X\circ \mathfrak{T}_X(S))$ is the intersection $\bigcap_{0\in U\in \mathfrak{T}_X(S)}\phi(U)$. Since $\phi(0)=0$, this set is equal to $\bigcap_{0\in V \in \phi_*(\mathfrak{T}_X(S))}V=\mathfrak{S}_Y(\phi_*(\mathfrak{T}_X(S)))$, which implies that $\phi(S)$ is a subspace of $Y$. By the same argument for $\phi^{-1}$, we conclude that $\phi$ induces an isomorphism between the lattices of subspaces $\sigma_K(X)$ and $\sigma_L(Y)$.\\
For an element $a$ of $X$ and a subspace $S$ of $X$, we show that $\phi(a+S)\subset \phi(a)+\phi(S)$ by a contradiction. Assume that there exists an element $x\in S$ such that $y=\phi(a+x)\in\phi(a+S)
$ does not belong to $\phi(a)+\phi(S)$. Since $\phi(S)=\mathfrak{S}_Y(\phi_*(\mathfrak{T}_X(S)))$ holds and the translation map $y\mapsto  y+\phi(a)$ is a homeomorphism, we have
\begin{align*}
\phi(a)+\phi(S)=\bigcap_{\phi(a)\in V\in\phi_*(\mathfrak{T}_X(S))}V=\bigcap_{a\in U\in \mathfrak{T}_X(S)}\phi(U).
\end{align*}
Thus, there exists an open neighborhood $U\in \mathfrak{T}_X(S)$ of $a$ such that $a+x\not\in U$. On the other hand, the open set $U$ is $S$-invariant, that is, $U=U+S$ holds, which implies that $U$ is an open neighborhood of $a+x$, and this is a contradiction. Therefore, we obtain $\phi(a+S)\subset \phi(a)+\phi(S)$.
By combining the same argument for $\phi^{-1}$, we conclude that $\phi(a+S)=\phi(a)+\phi(S)$ holds for every $a\in X$ and $S\in\sigma_K(X)$. In particular, the bijection $\phi$ sends parallel lines of $X$ to that of $Y$. By applying the fundamental theorem of affine geometry (Theorem \ref{lemmaFTAG}) for $\phi$, we deduce that $\phi:X\rightarrow Y$ is a semiaffine map. Namely, there exists a field isomorphism $\psi:K\rightarrow L$ and $\phi$ is a composition of $\psi$-semilinear isomorphism and a translation. Since we define $\phi$ so that $\phi(0)=0$, the translation is trivial and $\phi$ is a $\psi$-semilinear isomorphism.\\
To see that $\psi:K\rightarrow L$ is a homeomorphism, we fix a non-zero element $x_0\in X$. Since $\phi_*(T^{\max}_X)=T^{\max}_Y$, the map $\phi:(X,T^{\max}_X)\rightarrow(Y,T^{\max}_Y)$ is a homeomorphism. Thus, the restriction of $\phi$ to the subspace $\langle x_0\rangle$ is a homeomorphism between $\langle x_0\rangle$ and $\langle\phi(x_0)\rangle$, where $\langle x_0\rangle$ and $\langle\phi(x_0)\rangle$ are 1-dimensional subspaces generated by $x_0$ and $\phi(x_0)$, respectively.
 By the identifications $K\ni \alpha \mapsto \alpha\cdot x_0\in \langle x_0\rangle$ and $L\ni \beta \mapsto \beta \cdot\phi(x_0)\in\langle\phi(x_0)\rangle$, the subspaces $\langle x_0\rangle$ and $\langle\phi(x_0)\rangle$ are homeomorphic to $K$ and $L$, respectively, and the restriction coincides with $\psi$. Therefore, the map $\psi$ is an isomorphism between topological fields $K$ and $L$. \par
Next, we consider the case when ${\phi_0}_*$ is $C\circ \Phi$. In this case, by Lemma \ref{Hartmanis}, this case happens when the cardinality of $Y$ is finite. Thus $Y$ is finite-dimensional and the coefficient field $L$ is a finite discrete topological field since we assume that $L$ is a Hausdorff space. Moreover, since $L$ is a discrete topological field and $Y$ is a finite set, the discrete topology is the only Hausdorff vector topology on $Y$.\\
Now, it is known that if $Y$ admits only one Hausdorff vector topology, the lattice $\tau_L(Y)$ of vector topologies is isomorphic to the lattice $\sigma_L(Y)$ of $L$-subspaces of $Y$ by $\mathfrak{S}_Y$ and $\mathfrak{T}_Y$ (see \cite[Lemma 2.9]{T.Aoy}). 
Thus, for an arbitrary element $T$ of $\tau_L(Y)$, there exists a unique subspace $S$ such that 
$
T=\{ U+S \mid U\in T^{\max}_Y\}
$
holds. Moreover, since $T^{\max}_Y$ is a discrete topology, $T$ is generated by a base $\{y+S \mid y\in Y\}$, which implies that $C(T)=T$. Therefore, the map $C$ preserves the set $\tau_L(Y)$,  and the map $\Phi':\Sigma(X)\rightarrow\Sigma(Y)$ defined by $C\circ\Phi$ also satisfies the condition of $\Phi$. Namely, $\Phi':\Sigma(X)\rightarrow\Sigma(Y)$ is a lattice isomorphism and $\Phi'(\tau_K(X))=\tau_L(Y)$. Thus the same argument in the first part of this proof for $\Phi'$ and $\phi_0$ shows that there exists the required triple $(\psi,\phi,y_0)$, which also satisfies $(\phi+y_0)_*=\Phi'=C\circ \Phi$. \par
We complete the proof by showing the uniqueness of the triple $(\psi,\phi,y_0)$. Let $(\psi_1,\phi_1,y_1)$ and $(\psi_2,\phi_2,y_2)$ be two triples satisfying the condition. By the same argument in the last part of the proof of Lemma \ref{Hartmanis}, we have $(\phi_1+y_1)=(\phi_2+y_2)$. By substituting $0$, we obtain $y_1=y_2$ and $\phi_1=\phi_2$. Since $\psi_1$ and $\psi_2$ satisfy $\psi_i(\alpha)\phi_i(x)=\phi_i(\alpha \cdot x),~i=1,2$ for every $\alpha\in K,~x\in X$, The equality $\phi_1=\phi_2$ implies $\psi_1=\psi_2$.
\end{proof}

By using Theorem \ref{main thm}, we study the group of lattice automorphisms which preserves the lattice of vector topologies. We introduce notations:

\begin{dfn}
Let $X$ be a vector space over a topological field $K$ with $\dim_K X\geq2$. We denote by ${\rm Aut}(\Sigma(X),\tau_K(X))$, the subgroup of the group of lattice automorphisms on $\Sigma(X)$ preserving $\tau_K(X)$:
\begin{align*}
{\rm Aut}(\Sigma(X),\tau_K(X))=\{\Phi \in {\rm Aut}(\Sigma(X))\mid \Phi(\tau_K(X))=\tau_K(X)\}.
\end{align*}
We also denoted by $\Gamma L_h(X)$, the subgroup of the group of semilinear automorphisms on $X$ defined by
\begin{align*}
\Gamma L_h(X)=\{\phi:X\rightarrow X \mid \text{$\psi$-semilinear for a homeomorphic isomorphism $\psi:K\rightarrow K$}\}.
\end{align*}
\end{dfn}

As a consequence of Theorem \ref{main thm}, we obtain the following theorem: 
\begin{thm}\label{decompose}
Let $X$ be a vector space over a topological field with $\dim_KX\geq2$. If the cardinality of $X$ is finite, then the group ${\rm Aut}(\Sigma(X),\tau_K(X))$ is isomorphic to $(X\rtimes\Gamma L_h(X))\times\mathbb{Z}/2\mathbb{Z}$ by
\begin{align*}
F:(X\rtimes\Gamma L_h(X))\times\mathbb{Z}/2\mathbb{Z}\ni((x,\phi),\epsilon)\mapsto C^{\epsilon}\circ(\phi+x)_* \in {\rm Aut}(\Sigma(X),\tau_K(X)),
\end{align*}
where $C^0={\rm id}_{\Sigma(X)},~C^1=C$ and the group operation of $X\rtimes\Gamma L_h(X)$ is defined by $(x_1,\phi_1)\cdot(x_2,\phi_2)=(x_1+\phi_1(x_2),\phi_1\circ\phi_2)$.\\
If the cardinality of $X$ is infinite, then ${\rm Aut}(\Sigma(X),\tau_K(X))$ is isomorphic to $X\rtimes\Gamma L_h(X)$ by
\begin{align*}
G:X\rtimes\Gamma L_h(X)\ni(x,\phi)\mapsto(\phi+x)_* \in {\rm Aut}(\Sigma(X),\tau_K(X)).
\end{align*}
\end{thm}

\begin{proof}
Since every element of $\Gamma L_h(X)$ is a $\psi$-semilinear map for a homeomorphism $\psi:K\rightarrow K$, the restriction $\phi_*\restriction_{\tau_K(X)}$ is a map from $\tau_K(X)$ to itself. Also, translations $X\ni x\mapsto x+a\in X$ are homeomorphisms with respect to every vector topology. Thus the image of $\tau_K(X)$ by $(\phi+x)_*$ is again $\tau_K(X)$. Moreover, in the proof of Theorem \ref{main thm}, the complement map $C$ is an  identity map on $\tau_K(X)$ when $X$ is a finite set. Thus $F(x,\phi,\epsilon)$ and $G(x,\phi)$ are elements of ${\rm Aut}(\Sigma(X),\tau_K(X))$. \\
For $(x_1,\phi_1),(x_2,\phi_2)\in X\rtimes\Gamma L_h(X)$, since $\phi_1$ preserves the addition, we have
\begin{align*}
(\phi_1+x_1)_*\circ(\phi_2+x_2)_*&=\{(\phi_1+x_1)\circ(\phi_2+x_2)\}_*\\
                                             &=\{\phi_1\circ\phi_2 +x_1+\phi_1(x_2)\}_*.  
\end{align*}
Moreover, if $X$ is a finite set, the induced map $f_*:\Sigma(X)\rightarrow\Sigma(X)$ by a bijection $f:X\rightarrow X$ commutes with $C$ from the following equality:
\begin{align*}
f_*\circ C(T)&=f_*(\{X\setminus U\mid U\in T\})\\
                  &=\{f(X\setminus U) \mid U\in T\}\\
                  &=\{X\setminus f(U)\mid U\in T\}\\
                  &=C\circ f_*(T).
\end{align*}
Therefore, $F$ and $G$ are group homomorphisms. Lastly, by Theorem \ref{main thm}, we obtain that two maps $F,G$ are bijections. 
\end{proof}

\section{Lattice of vector topologies and Hausdorff vector topologies}

Next, we consider the lattice $\Sigma_1(X)$ of all ${\rm T_1}$-topologies on $X$ as an analogy of Corollary \ref{cor main thm}. Namely, is it true that if there exists an isomorphism $\Phi:\Sigma_1(X)\rightarrow\Sigma_1(Y)$ between lattices of ${\rm T_1}$-topologies on vector spaces $X,Y$ over topological fields $K,L$, respectively, such that the image of $\tau_K(X)\cap\Sigma_1(X)$ is $\tau_L(Y)\cap\Sigma_1(Y)$, then $K$ and $L$ are isomorphic, and $X$ and $Y$ have the same dimension? \\
In \cite[Theorem 5]{J.Har} and \cite{D.Ell}, it is shown that the group of lattice automorphisms of $\Sigma_1(X)$ on an infinite set $X$ is isomorphic to the symmetric group of $X$. 
However, this question is negatively answered by the following example.
\begin{exa}\label{only compatible}
Let $X$ be the $\mathbb{R}$-vector space $\mathbb{R}^4$ and $Y$ be the $\mathbb{C}$-vector space $\mathbb{C}^2$. It is clear that $\mathbb{R}$ and $\mathbb{C}$ are not isomorphic, and $X$ and $Y$ have different dimensions. However, it is known that ${\rm T_1}$-vector topologies are Hausdorff and that every finite-dimensional vector space over a non-trivial complete valued field admits only one Hausdorff vector topology (see \cite[\S 2, No.3, Theorem 2]{N.Bou}). Thus, $\tau_K(X)\cap\Sigma_1(X)$ and $\tau_L(Y)\cap\Sigma_1(Y)$ consist of only one elements: the Euclidean topologies. Therefore, any homeomorphism $\phi$ between $\mathbb{R}^4$ and $\mathbb{C}^2$ with respect to the Euclidean topologies induces the map $\phi_*:\Sigma_1(X)\rightarrow\Sigma_1(Y)$ such that the image of $\tau_K(X)\cap\Sigma_1(X)$ is $\tau_L(Y)\cap\Sigma_1(Y)$.
\end{exa}
\par
The restriction of $\Phi$ in Theorem \ref{main thm} to $\tau_K(X)$ is a lattice isomorphism between $\tau_K(X)$ and $\tau_L(Y)$. 
Thus, it is natural to consider whether a lattice isomorphism between the lattices of vector topologies implies $K\cong L$ and $X\cong Y$ or not. 
This is not true in general, because it is shown in \cite[Theorem 3.3]{T.Aoy}, that the lattice of vector topologies of $\mathbb{Q}^n$ is isomorphic to that of $\mathbb{R}^n$, which is also isomorphic to $(\sigma_{\mathbb{R}}(\mathbb{R}^n), \supset)$.\\
However, if we consider an isomorphism between lattices of vector topologies which preserves Hausdorff vector topologies, we recover algebraic aspects of Corollary \ref{cor main thm}.
\begin{thm}\label{comatible topologies with Hausdroff}
Let $K$ and $L$ be topological fields and let $X$ and $Y$ be vector spaces over $K$ and $L$, respectively, with $\dim_K(X)\geq3$. Suppose that there exists a lattice isomorphism $\Phi:\tau_K(X)\rightarrow\tau_L(Y)$ such that $\Phi(\tau_K^H(X))=\tau_L^H(Y)$. Then, $K$ is isomorphic to $L$ as fields {\rm(}not necessarily as topological fields{\rm)}, and $X$ and $Y$ have the same dimension.
\end{thm}

\begin{proof}
For a non-negative integer $d$, we denote by $\sigma_K^d(X)$ and $\sigma_L^d(Y)$, the set of $d$-dimensional subspaces of $X$ and $Y$, respectively. We also denote by $\sigma_K^{<\infty}(X)$ and $\sigma_L^{<\infty}(Y)$, the sublattice of $\sigma_K(X)$ and $\sigma_L(Y)$ consisting of finite-dimensional subspaces of $X$ and $Y$, respectively. By the fundamental theorem of projective geometry (Theorem \ref{lemmaFTPG}) and Remark \ref{remarkFTPG}, it is enough to construct a lattice isomorphism between $\sigma_K^{<\infty}(X)$ and $\sigma_L^{<\infty}(Y)$. We define two maps $F:\sigma_K(X)\rightarrow\sigma_L(Y)$ and $G:\sigma_L(Y)\rightarrow\sigma_K(X)$ as the following compositions:
\begin{align*}
F:\sigma_K(X)\xrightarrow{\mathfrak{T}_X}\tau_K(X)\xrightarrow{\Phi}\tau_L(Y)\xrightarrow{\mathfrak{S}_Y}\sigma_L(Y),\\
G:\sigma_L(Y)\xrightarrow{\mathfrak{T}_Y}\tau_L(Y)\xrightarrow{\Phi^{-1}}\tau_K(X)\xrightarrow{\mathfrak{S}_X}\sigma_K(X).
\end{align*}
By definition, $\mathfrak{T}_X,\mathfrak{T}_Y,\mathfrak{S}_X,\mathfrak{S}_Y$ invert the inclusion $\subset$. Thus, $F$ and $G$ preserve the inclusion order. By (2) of Lemma \ref{composition}, the topology $\Phi\circ\mathfrak{T}_X(S)$ is included in $\mathfrak{T}_Y\circ\mathfrak{S}_Y\circ\Phi\circ\mathfrak{T}_X(S)$. Thus, we have
\begin{align*}
G\circ F(S)&=\mathfrak{S}_X\circ\Phi^{-1}\circ\mathfrak{T}_Y\circ\mathfrak{S}_Y\circ\Phi\circ\mathfrak{T}_X(S)\\
&\subset \mathfrak{S}_X\circ\Phi^{-1}\circ\Phi\circ\mathfrak{T}_X(S)\\
&=\mathfrak{S}_X\circ\mathfrak{T}_X(S).
\end{align*}
By (1) of Lemma \ref{composition}, we have $\mathfrak{S}_X\circ\mathfrak{T}_X(S)=S$, and we obtain $G\circ F(S)\subset S$. The same argument shows that $F\circ G(S)\subset S$ holds.
Now, we show by induction on the non-negative integer $d$, that the restrictions of $F$ and $G$ to $\sigma_K^d(X)$ and $\sigma_L^d(Y)$ are bijections.\par
(The base case)~For the case of $d=0$, the $0$-dimensional subspace is mapped to $T^{\max}_X$ by $\mathfrak{T}_X$. Since $T^{\max}_X$ is the maximum element, $\Phi$ sends $T^{\max}_X$ to $T^{\max}_Y$, and Hausdorff topology $T^{\max}_Y$ is mapped to $0$-dimensional subspace of $Y$ by (3) of Lemma \ref{composition}. Thus, the restriction of $F$ to $\sigma_K^0(X)$ is a bijection between $\sigma_K^0(X)$ and $\sigma_L^0(Y)$. \\
We show that the case of $d=1$ also holds. First, assume that two different $1$-dimensional subspaces $S_1,S_2$ of $X$ are mapped to the same subspace $S'=F(S_1)=F(S_2)$. Then, $G\circ F(S_1)=G\circ F(S_2)$ is included in $S_1\cap S_2=\{0\}$, which implies that $\Phi^{-1}\circ\mathfrak{T}_Y(S')$ is a Hausdorff topology by (3) of Lemma \ref{composition}. Since $\Phi$ preserves Hausdorff vector topologies, $\mathfrak{T}_Y(S')$ is also a Hausdorff topology. Thus, by (3) of Lemma \ref{composition}, $S'$ is $0$-dimensional, but this contradicts that $S_1$ is mapped to a non-Hausdorff topology by $\Phi\circ\mathfrak{T}_X$. Hence $F\restriction_{\sigma_K^1(X)}$ is an injection, and the same argument shows that $G\restriction_{\sigma_L^1(Y)}$ is also an injection.  
Next, assume that $\dim_KF(S)\geq2$ for $S\in\sigma_K^1(X)$. Then, we have two different subspaces $S'_1,S'_2\in\sigma_L^1(Y)$ which are subspaces of $F(S)$. Since $G$ preserves the inclusion and $G\circ F(S)\subset S$, we have $G(S'_1)$ and $G(S'_2)$ are subspaces of $S$. Since $\Phi$ preserves Hausdorff vector topologies, $G(S'_1)$ and $G(S'_2)$ are not $0$-dimensional, and thus they coincide with $S$. This contradicts that $G\restriction_{\sigma_L^1(Y)}$ is injective. Thus, the restriction of $F$ is a map from $\sigma_K^1(X)$ to $\sigma_L^1(Y)$. Also, the same argument shows that $G\restriction_{\sigma_L^1(Y)}:\sigma_L^1(Y)\rightarrow\sigma_K^1(X)$. This implies that $G\circ F(S)$ and $F\circ G(S')$ is $1$-dimesional subspaces of $S$ and $S'$, respectively for $S\in\sigma_K^1(X),~S'\in\sigma_L^1(Y)$. Thus, $G\circ F$ and $F\circ G$ are identity maps when they are restricted to $1$-dimensional subspaces.\par
(The inductive step)~We assume that the statement holds for $d=0,1,\cdots, d'$ and show that the  case of $d'+1(\geq2)$ also holds.\\
First, we note that $F$ and $G$ do not send a $(d'+1)$-dimensional subspaces to a lower dimensional subspace. Otherwise, if $S\in\sigma_K^{d'+1}(X)$ is sent to a subspace lower than $d'+1$, we have two different $d'$-dimensional subspaces of $S$. They are sent to the same subspace $F(S)$ since $F$ preserves the inclusion. This is a contradiction to our assumption of the induction. \\
Next, we show that $F\restriction_{\sigma_K^{d'+1}(X)}$ is injective. Assume that two different elements $S_1,S_2\in\sigma_K^{d'+1}(X)$ are mapped to the same subspace $S'=F(S_1)=F(S_2)$. Then, we take two different $d'$-dimensional subspaces $S'_1$ and $S'_2$ of $S'$ since $\dim_L(S')\geq d'+1$ holds. Then, $G(S'_1),G(S'_2)$ are subspaces of $G\circ F(S_1)$ and $G\circ F(S_2)$. Moreover, they are subspaces of $S_1\cap S_2$, whose dimension is not greater than $d'$. Thus, $G(S'_1)$ and $G(S'_2)$ coincide, which contradicts the assumption of the induction. Combining with the same argument for $G$ shows that $F\restriction_{\sigma_K^{d'+1}(X)}$ and $G\restriction_{\sigma_L^{d'+1}(Y)}$ are injective. \\
Lastly, we show that $\dim_LF(S)$ is equal to $d'+1$ for $S\in\sigma_K^{d'+1}(X)$. We have already proved that $\dim_LF(S)$ is not less than $d'+1$. Thus, we assume that $\dim_LF(S)$ is greater than $d'+1$. Then, we have two different $(d'+1)$-dimensional subspaces $S'_1,S'_2$ of $F(S)$. They are mapped to subspaces of $S$ by $G$, and their dimensions are $d'+1$ since $G$ does not decrese dimensions. Thus, $G(S'_1)$ and $G(S'_2)$ coincide with $S$, which is a contradiction since $G\restriction_{\sigma_L^{d'+1}(Y)}$ is injective. Therefore, $F\restriction_{\sigma_K^{d'+1}(X)}$ is a map from $\sigma_K^{d'+1}(X)$ to $\sigma_L^{d'+1}(Y)$, and by the same argument, we have  $G\restriction_{\sigma_L^{d'+1}}:\sigma_L^{d'+1}(Y)\rightarrow\sigma_K^{d'+1}(X)$. Thus, the restriction of $F\circ G$ and $G\circ F$ to $(d'+1)$-dimensional subspaces are identity maps, which shows that the case of $d=d'+1$ also holds.
\end{proof}

We end the papar by giving an example, which shows that the map $\psi:K\rightarrow L$ in Theorem \ref{comatible topologies with Hausdroff} may not be continuous. 

\begin{exa}\label{not necessarily continuous}
We fix a prime number $p$. Let $\mathbb{C}_p$ be the completion of the algebraic closure $\overline{\mathbb{Q}_p}$ of the field of $p$-adic numbers $\mathbb{Q}_p$. The $\mathbb{C}_p$ has a complete valuation $|\cdot|_p$. Now, let $K$ be $\mathbb{C}_p$ and $L$ be the complex field $\mathbb{C}$ with the standard absolute value. Let $X,Y$ be 3-dimensional vector spaces defined by $K^3,L^3$, respectively. It is known that the lattice of vector topologies on a finite-dimensional vector space over a non-trivial complete valuation is isomorphic to the lattice of its subspaces (See \cite{E.Ste,T.Aoy}).  Thus, we obtain two isomorphisms:
\begin{align*}
\tau_K(X)\cong\sigma_K(X),~\tau_L(Y)\cong\sigma_L(Y).
\end{align*}
Moreover, it is known that $\mathbb{C}_p$ and $\mathbb{C}$ are isomorphic as fields. Thus, the lattices of subspaces are isomorphic:
\begin{align*}
\sigma_K(X)\cong \sigma_L(Y).
\end{align*}
Therefore, we obtain a lattice isomorphism $\tau_K(X)\cong\tau_L(Y)$ denoted by $\Phi$. It is shown in \cite[\S 2, No.3, Theorem 2]{N.Bou}, that two subsets $\tau_K^H(X)$ and $\tau_L^H(Y)$ are one point sets $\{T^{\max}_X\},~\{T^{\max}_Y\}$ for non-trivial complete valued fields $K,L$. Thus, the image of $\tau_K^H(X)$ by $\Phi$ coincides with $\tau_L^H(Y)$. However, the fields $K$ and $L$ are not homeomorphic by any field isomorphisms.
\end{exa}

\section*{Acknowledgments}
The author would like to express his gratitude to Prof. Ken'ichi Ohshika and Prof. Shinpei Baba for valuable comments. 
He also would like to thank the Research Institute for Mathematical
Sciences (RIMS) for giving him opportunities of research presentations on results in this paper at Kyoto University.

\end{document}